\documentclass[12pt,a4paper]{article}

\usepackage[latin2]{inputenc}
\usepackage{amssymb}
\usepackage{paralist}
\usepackage{amsmath, calc}
\usepackage{amsfonts}
\usepackage{amsthm}
\usepackage{fullpage}
\usepackage{enumerate}
\usepackage{paralist}

\newtheorem{thm}{Theorem}
\newtheorem{lem}{Lemma}
\newtheorem{prob}{Problem}
\newtheorem{conj}{Conjecture}
\newtheorem{cor}{Corollary}
\newtheorem{rem}{Remark}
\newtheorem{prop}{Proposition}

\begin{document}

\title{On a conjecture of Erd\H{o}s about sets without $k$ pairwise coprime
integers}
\author{S\'andor Z. Kiss \thanks{Institute of Mathematics, Budapest
University of Technology and Economics, H-1529 B.O. Box, Hungary;
kisspest@cs.elte.hu;
This author was supported by the National Research, Development and Innovation Office NKFIH Grant No. K115288.}, Csaba
S\'andor \thanks{Institute of Mathematics, Budapest University of
Technology and Economics, H-1529 B.O. Box, Hungary, csandor@math.bme.hu.
This author was supported by the OTKA Grant No. K109789. This paper was supported
by the J\'anos Bolyai Research Scholarship of the Hungarian Academy of Sciences.},
Quan-Hui Yang \thanks{School of Mathematics and Statistics, Nanjing University of Information Science and Technology, Nanjing 210044, China; yangquanhui01@163.com; This author was supported by the National Natural Science
Foundation for Youth of China, Grant No. 11501299, the Natural
Science Foundation of Jiangsu Province, Grant Nos.
BK20150889,~15KJB110014 and the Startup Foundation for Introducing
Talent of NUIST, Grant No. 2014r029.}
}
\date{}
\maketitle

\begin{abstract}
\noindent Let $\mathbb{Z}^{+}$ be the set of positive integers.
Let $C_{k}$ denote all subsets of $\mathbb{Z}^{+}$ such that
neither of them contains $k + 1$ pairwise coprime integers and
$C_k(n)=C_k\cap \{1,2,\ldots,n\}$. Let $f(n, k) = \text{max}_{A
\in C_{k}(n)}|A|$, where $|A|$ denotes the number of elements of
the set $A$. Let $E_k(n)$ be the set of positive integers not
exceeding $n$ which are divisible by at least one of the primes
$p_{1}, \dots{}, p_{k}$, where  $p_{i}$ denote the $i$th prime
number. In 1962, Erd\H{o}s conjectured that $f(n, k) = |E(n,k)|$
for every $n \ge p_{k}$. Recently Chen and Zhou proved some
results about this conjecture. In this paper we solve an open
problem of Chen and Zhou and prove several related results about
the conjecture.


\end{abstract}

%\textit{2000 AMS \ Mathematics subject classification number}: 11B75.
%\textit{Key words and phrases}: combinatorial number theory, extremal sequences, pairwise coprime integers, Erd\H{o}s conjecture.

\section{Introduction}

Let $\mathbb{Z}^{+}$ be the set of positive integers and $p_{i}$
denote the $i$th prime number. For a set $A\subseteq
\mathbb{Z}^{+}$, we define $A(n)=A\cap \{1,2,\ldots,n\}$. Let
$C_{k}$ denote all subsets of $\mathbb{Z}^{+}$ such that neither
of them contains $k + 1$ pairwise coprime integers. Let $f(n, k) =
\text{max}_{A \in C_{k}(n)}|A|$, where $|A|$ denotes the number of
elements of $A$. Let $E_{k}$ be the set of positive integers which
are divisible by at least one of the primes $p_{1},p_2, \ldots{},
p_{k}$. Clearly $E_k(n)\in C_k(n)$. Hence $f(n,k) \ge |E_k(n)|$
for all integers $n,k$ and $f(n, k) = n = |E_k(n)| + 1$ if $n <
p_{k}$. In 1962, Erd\H{o}s \cite{Erd}, \cite{Erc} conjectured that
$f(n, k) = |E_k(n)|$ for every $n \ge p_{k}$. It can be proved
easily that this conjecture holds for $k = 1$ and $k = 2$. For $k
= 3$, the conjecture was proved by Choi \cite{Choi} and
independently by Szab\'o and T\'oth \cite{Cst}. Choi \cite{Choi}
also showed that for $n \ge 150$, if $A \in C_{3}(n)$ and $|A| =
|E_3(n)|$, then $A = E_3(n)$. Later, M\'ocsy \cite{Moc} proved
that the Erd\H{o}s' conjecture holds for $k = 4$ and Ahlswede and
Khachatrian \cite{Ahk} disproved the conjecture for $k = 212$. In
\cite{des} Chen and Zhou gave a new proof for $k = 4$ and
disproved the conjecture for $k = 211$. Beyond these, they also
proved some related results and posed the following problem (see
\cite[Problem 3]{des}).

{\noindent{\bf Chen and Zhou's problem.}\label{prob1} Is
$\limsup_{k \rightarrow \infty}\text{sup}_{n \ge 1}(f(n, k) -
|E_k(n)|) < +\infty$?}

In this paper, we answer this problem.

\begin{thm}
$\limsup_{k \rightarrow \infty} \sup_{n \ge 1}(f(n, k) - |E_k(n)|)
= +\infty$.
\end{thm}

Theorem 1 implies that the conjecture of Erd\H{o}s is false for
infinitely many $k$, which also solves Problem 2 in \cite{des}.

%The proof of Theorem 1 suggests that Conjecture 4 in \cite{des}
%may not hold.
%\begin{conj}
%The set of $k$ for which Erd\H{o}s' conjecture is false has
%positive density.
%\end{conj}

On the other hand, motivated by the method of Pintz \cite{Ptc}
about the gaps between primes, we pose the following problem.

\begin{prob}
Find a suitable function $g(k) \rightarrow \infty$ such that
\[
\limsup_{k \rightarrow \infty}\frac{{\rm{sup}}_{n \ge 1}(f(n, k) -
|E_k(n)|)}{g(k)} > 0.
\]
\end{prob}

We also pose the following natural problem.
\begin{prob}
Is it true that ${\rm{sup}}_{n \ge 1}(f(n, k) - |E_k(n)|) <
+\infty$ holds for every positive integer $k$?
\end{prob}

For an infinite positive integers set $A$, we define
$$\bar{d}(A)=\limsup_{n\rightarrow +\infty}\frac{|A(n)|}{n},\quad \underline{d}(A)=\liminf_{n\rightarrow
+\infty}\frac{|A(n)|}{n}.$$ If $\bar{d}(A)=\underline{d}(A)$, then
we define $d(A):=\bar{d}(A)=\underline{d}(A)$.

Now we have the following conjecture.

\begin{conj}
If $A \in C_{k}$ and $A \not\subseteq E_{k}$, then $\bar{d}(A) <
d(E_{k}) = 1 - \frac{\varphi(p_{1}p_2\cdots p_{k})}{p_{1}p_2
\cdots p_{k}}$.
\end{conj}

Let $B_{k}$ denote the set of positive integers $m$ such that $m$
is divisible by at least one of the primes $p_{1},p_2, \ldots,
p_{k-1},p_{k+1}$. Clearly $B_{k}(n) \in C_{k}(n)$, $B_{k}(n)
\not\subseteq E_{k}(n)$ if $n \ge p_{k+1}$, and
\begin{eqnarray*}
B_{k}(n)& =& \frac{p_{1}p_2 \cdots p_{k-1} \cdot p_{k+1} -
\varphi(p_{1}p_2\cdots p_{k-1} \cdot p_{k+1})}{p_{1}p_2 \cdots
p_{k-1} \cdot p_{k+1}}n + O(1)
\\
&= &\frac{\Big(p_{1}p_2 \cdots p_{k} - \varphi(p_{1}\cdots
p_{k})\Big)\cdot p_{k+1} - \varphi(p_{1}p_2\cdots p_{k-1})(p_{k+1}
- p_{k})}{p_{1}p_2\cdots p_{k}p_{k+1}}n + O(1).
\end{eqnarray*}

We guess that Conjecture 2 can be sharpened in the following quantitative form.

\begin{conj}\label{170410} For any positive integer $k$, let $P_k=p_1p_2\cdots
p_k$, where $p_i$ denotes the $i$th prime. If $A\in C_k,~a\in
A\setminus E_k$, then
$$|A(n)|\le
\frac{(P_k-\varphi(P_k))a-\varphi(P_{k-1})(p_{k+1}-p_k)}{P_k\cdot
a}n+c_k,$$ where $c_k$ is a constant only depending on $k$.
\end{conj}

In this paper, we prove that Conjecture \ref{170410} holds for
$k=1,2,3$.

%\begin{conj}
%If $A \in C_{k}(n)$ and $a \in A \setminus E_{k}$ then
%\[
%|A(n)| \le \frac{\Big(p_{1} \cdots p_{k} - \varphi(p_{1}\cdots p_{k}\Big)\cdot %a - \varphi(p_{1}\cdots p_{k-1})(p_{k+1} - p_{k})}{p_{1}\cdots p_{k}\cdot a}n +% C_{k}.
%\]
%\end{conj}

%We prove that Conjecture 3 is true.

\begin{thm}
Conjecture 2 is true for $k = 1, 2, 3$.
\end{thm}

\begin{cor}
\begin{itemize}
\item[(1)] If $A\in C_1$ and $a\in A\setminus E_1$, then
$\bar{d}(A)\le \frac{a-1}{2a}$; \item[(2)] If $A\in C_2$ and $a\in
A\setminus E_2$, then $\bar{d}(A)\le \frac{4a-2}{6a}$; \item[(3)]
If $A\in C_3$ and $a\in A\setminus E_3$, then $\bar{d}(A)\le
\frac{22a-4}{30a}$.
\end{itemize}
\end{cor}

\begin{cor}
Conjecture 1 is true for $k = 1, 2, 3$.
\end{cor}

 If $f$ and $g$ are real
functions, then $f \ll g~(f\gg g)$ means that there is an absolute
positive constant $c$ such that $f(x)\le cg(x)~(f(x)\ge cg(x))$
for all sufficiently large $x$.

We give lower bounds of the difference of $|E_k(n)|$ and
$\displaystyle\max_{\substack{A \in C_{1}(n) \\ A \not\subseteq
E_1(n)}}|A|$ .

\begin{thm} (1). $\displaystyle|E_1(n)| - \max_{\substack{A \in C_{1}(n) \\ A
\not\subseteq E_1(n)}}|A| \gg \frac{n}{(\log\log n)^{2}};$

(2). $\displaystyle|E_2(n)| - \max_{\substack{A \in C_{2}(n) \\ A
\not\subseteq E_2(n)}}|A| \gg \frac{n}{(\log\log n)^{5}}$.
\end{thm}

Finally we show that the bound in Theorem 3 is nearly best
possible.

\begin{thm}
For any positive integer $k$, there exists a set of positive
integers $A$ such that $A \in C_{k}(n)$, $A \not\subseteq E_k(n)$
and
\[
|E_k(n)|-|A|\ll_k \frac{n}{\log\log n}.
\]
\end{thm}

\begin{rem} By Theorems 3 and 4, we have
$$\frac{n}{(\log\log n)^2}\ll \min_{\substack{A\subseteq C_1(n)\\
A\not\subseteq E_1(n)}}(|E_1(n)|-|A|)\ll \frac{n}{\log\log n},$$
$$\frac{n}{(\log\log n)^5}\ll \min_{\substack{A\subseteq C_2(n)\\
A\not\subseteq E_2(n)}}(|E_2(n)|-|A|)\ll \frac{n}{\log\log n}.$$
\end{rem}

\begin{prob}
Is it true that
\[
|E_1(n)| - \max_{\substack{A \in C_{1}(n) \\ A \not\subseteq
E_1(n)}}|A| \gg \frac{n}{\log\log n}?
\]
\end{prob}

We think that the statement of Theorem 3 is true in general.

\begin{conj} For every $k \ge 1$ there exists a constant $c_{k}$ and a positive integer $n_{k}$ such that for $n \ge n_{k}$
\[
|E_k(n)| - \max_{\substack{A \in C_{k}(n) \\ A \not\subseteq
E_k(n)}}|A| \gg \frac{n}{(\log\log n)^{c_{k}}}.
\]
\end{conj}

\section{Proof of Theorem 1}

First we will prove the following lemma.

\begin{lem}\label{lem1}
Let $l$ be a positive integer. Then there exist infinitely many positive integers $t$ such that $p_{t}p_{t+2l} > p_{t+2l-1}^{2}$.
\end{lem}

\begin{proof} Let $d_{n}$ denote the difference between consecutive primes, i.e., $d_{n} = p_{n+1} - p_{n}$. In \cite{Ptc}, J\'anos Pintz proved that
%\[
%\limsup_{t \rightarrow \infty}\frac{d_{t+2l-1}}{\text{max}\{d_{t},d_{t+1}, \dot%s{}, d_{t+2l-2}, d_{t+2l}, \dots{}, d_{t+4l-2}\}} = +\infty.
%\]
%Since
%\[
%\frac{d_{t+2l-1}}{\text{max}\{d_{t},d_{t+1}, \dots{}, d_{t+2l-2}, d_{t+2l}, \do%ts{}, d_{t+4l-2}\}} \le \frac{d_{t+2l-1}}{\text{max}\{d_{t},d_{t+1}, \dots{}, d%_{t+2l-2}\}},
%\]
%therefore
\[
\limsup_{t \rightarrow \infty}\frac{d_{t+2l-1}}{\text{max}\{d_{t},d_{t+1}, \dots{}, d_{t+2l-2}\}} = +\infty.
\]
It follows that there exist infinitely many $t$ such that
\[
\frac{d_{t+2l-1}}{\text{max}\{d_{t},d_{t+1}, \dots{}, d_{t+2l-2}\}} \ge 8l^{2}.
\]
In this case we have
\begin{eqnarray}
p_{t+2l-1}^{2} &\le & (p_{t} + 2l\text{max}\{d_{t},d_{t+1}, \dots{}, d_{t+2l-2}\})^{2} \\
\nonumber &\le & p_{t}^{2} + 4l^{2}p_{t}\text{max}\{d_{t},d_{t+1}, \dots{}, d_{t+2l-2}\} + 4l^{2}(\text{max}\{d_{t},d_{t+1}, \dots{}, d_{t+2l-2}\})^{2}.
\end{eqnarray}
On the other hand we have
\begin{eqnarray}
p_{t}p_{t+2l} &\ge& p_{t}(p_{t} + d_{t+2l-1}) \ge p_{t}(p_{t} + 8l^{2}\text{max}\{d_{t},d_{t+1}, \dots{}, d_{t+2l-2}\}) \\
\nonumber &= &p_{t}^{2} + 8l^{2}p_{t}\text{max}\{d_{t},d_{t+1}, \dots{}, d_{t+2l-2}\}.
\end{eqnarray}
In view of (1) and (2) it is enough to prove that
\begin{eqnarray*}
&&p_{t}^{2} + 4l^{2}p_{t}\text{max}\{d_{t},d_{t+1}, \dots{}, d_{t+2l-2}\}
+ 4l^{2}(\text{max}\{d_{t},d_{t+1}, \dots{}, d_{t+2l-2}\})^{2}\\
&\le& p_{t}^{2} + 8l^{2}p_{t}\text{max}\{d_{t},d_{t+1}, \dots{}, d_{t+2l-2}\},
\end{eqnarray*}
which is equivalent to
\[
\text{max}\{d_{t},d_{t+1}, \dots{}, d_{t+2l-2}\} \le p_{t}.
\]
The latter statement is obviously holds in view of the fact that $d_{t+i} = O(t^{2/3})$ holds for every $1 \le i \le l$.
\end{proof}

Now we are ready to prove Theorem 1. Let $l$ be a fixed positive
integer. By Lemma \ref{lem1}, there exists an infinitely many
integers $t$ such that $p_tp_{t+2l}>p_{t+2l-1}^2$,
$p_t^2>p_{t+2l}$ and $p_t>\frac{p_{t+2l}}{2}$. For any such an
integer $t$, we choose $n$ such that $p_{t+2l-1}^2\le
n<p_tp_{t+2l}$ and let
\[
B = \left\{m: 1 \le m \le n, \hspace*{2mm} \left(m, \prod_{i=1}^{t-1}p_{i}\right) \ne 1\right\},
\]
\[
C = \{p_{t+i}p_{t+j}: 0 \le i \le l - 1, \hspace*{2mm} i < j \le 2l - 1\},
\]
\[
D = \{p_{t+i}: 0 \le i \le l - 1\} \cup \{p_{t+i}^{2}: 0 \le i \le l - 1\}.
\]
Now we verify that $E_{t+l-1}(n) = B \cup C \cup D$ is a suitable
decomposition. For any $m \in E_{t+l-1}(n) \setminus B$, let
$p_{t+i}$ be the smallest prime divisor of $m$. Clearly $0\le i\le
l-1$.

If there exists a prime $p_{t+j}$ with $p_{t+j}\not=p_{t+i}$ such
that $p_{t+j}\mid m$ and $p_{t+i}p_{t+j}\le n$, then by
$n<p_tp_{t+2l}$ and $p_{t+i}\ge p_t$, we have $p_{t+j}<p_{t+2l}$.
Hence, $i<j\le 2l-1$. Noting that $p_{t+i}p_{t+j}\mid m$ and
$2p_{t+i}p_{t+j}\ge 2p_t^2>p_tp_{t+2l}>n\ge m$, we have $m\in C$.

If there does not exist a prime divisor of $m$ different from
$p_{t+i}$, then by $p_t^3>p_t p_{t+2l}>n\ge m$ we have $m\in D$.

Define $D^{'} = \{p_{t+i}p_{t+j}: l \le i \le 2l - 2,
\hspace*{2mm} i < j \le 2l - 1\}$ and $A(n, t + l - 1) = B \cup C
\cup D^{'}$. Then
\[
C \cup D^{'} = \{p_{t+i}p_{t+j}: 0 \le i \le 2l - 2, \hspace*{2mm}
i < j \le 2l - 1\}.
\]
Clearly $C \cup D^{'}$ contains at most $l$ integers which are
pairwise coprime to each other and the set $B$ contains at most $t
- 1$ integers which are pairwise coprime to each other. It follows
that $A(n, t + l - 1)$ contains at most $t + l - 1$ integers which
are pairwise coprime to each other, which implies that $f(n, t + l
- 1) \ge |A(n, t + l - 1)|$. It is easy to see that $|D^{'}| = (l
- 1) + (l - 2) + \dots{} + 1 = \frac{l(l - 1)}{2}$. Thus we have
\begin{eqnarray*}
&&f(n, t + l - 1) - |E_{t+l-1}(n)| \ge
|A(n,t+l-1)|-|E_{t+l-1}(n)|\\
&=&|D^{'}| - |D| = \frac{l(l - 1)}{2} - 2l = \frac{l(l - 5)}{2}.
\end{eqnarray*}
This completes the proof of Theorem 1.

%The proof of Theorem 1 is completed.

\section{Preliminary Lemmas}

In this section, we present some lemmas for the proof of Theorem
2.

\begin{lem}\label{lem55} Let $k$ be a nonnegative integer. If
$|A\cap (30k+\{1,3,4,5,7,8,11,13\})|\ge 5$ or $|A\cap
(30k+\{17,19,22,23,25,26,27,29\})|\ge 5$, then $A$ contains $4$
pairwise coprime integers.\end{lem}

This Lemma follows from the proof of Lemmas 2 and 3 in \cite{Choi}
immediately.

\begin{lem}\label{lem7} Let $k$ be a nonnegative integer. If $7\in A$ and
$|A\cap \{7k+1,7k+2,\ldots,7k+6,7k+7\}|\ge 6$, then $A$ contains
$4$ pairwise coprime integers.\end{lem}

\begin{proof}
It is enough to prove that, if $|A\cap
\{7k+1,7k+2,\ldots,7k+6\}|\ge 5$, then $A$ contains $3$ pairwise
coprime integers. Without loss of generality, we may assume that
$7k+1$ is odd. Otherwise, we can replace $7k+a$ by $7k+7-a$ for
$a=1,2,\ldots,6$.

If $7k+3\in A$, by $|A\cap \{7k+1,7k+2,\ldots,7k+6\}|\ge 5$, then
either $\{7k+1,7k+2,7k+3\}\subseteq A$ or
$\{7k+3,7k+4,7k+5\}\subseteq A$, and so $A$ contains $3$ pairwise
coprime integers.

Now we suppose that $7k+3\not\in A$. Let $x\in \{7k+2,7k+4\}$
satisfy $3\nmid x$. Then $x,7k+1,7k+5\in A$ are pairwise coprime.
\end{proof}

\begin{lem}\label{lem13} Let $k$ be a nonnegative integer. If $13\in A$ and
$|A\cap \{13k+1,13k+2,\ldots,13k+12,13k+13\}|\ge 10$, then $A$
contains $4$ pairwise coprime integers.\end{lem}

\begin{proof} Write $S_1=\{13k+1,13k+3,13k+5,13k+7,13k+9,13k+11\}$ and $S_2=\{13k+2,13k+4,13k+6,13k+8,13k+10,13k+12\}$.
Clearly we have $|A\cap (S_1\cup S_2)|\ge 9$. Since $13$ is
coprime to all integers belong to $S_1$ and $S_2$, it is enough to
prove that $A\cap (S_1\cup S_2)$ contains $3$ pairwise coprime
integers. Without loss of generality, we assume that $2\nmid
13k+1$. Otherwise, we can replace $13k+a$ by $13k+13-a$ for
$a=1,2,\ldots,12$. We notice that $13k+1,13k+3,13k+5$ are pairwise
coprime and $13k+7,13k+9,13k+11$ are also pairwise coprime. Hence,
if $\{13k+1,13k+3,13k+5\}\subseteq A$ or
$\{13k+7,13k+9,13k+11\}\subseteq A$, then the result is true. Now
we assume that $|A\cap S_1|\le 4$. Noting that $|A\cap (S_1\cup
S_2)|\ge 9$ and $|A\cap S_2|\le 6$, we have $|A\cap S_1|\ge 3$.
Hence $|A\cap S_1|=3$ or $4$.

Case 1. $|A\cap S_1|=3$. In this case, we have~ $S_2\subseteq A$.
If there are two consecutive odd terms in $A$, then these two odd
integers and the even integer between them are pairwise coprime.
Since $13k+1,13k+5,13k+9$ are pairwise coprime and
$13k+3,13k+7,13k+11$ are also pairwise coprime, the result is true
if $\{13k+1,13k+5,13k+9\}\subseteq A$ or
$\{13k+3,13k+7,13k+11\}\subseteq A$. Hence, we only need to
consider the cases $A\cap S_1=\{13k+1,13k+5,13k+11\}$ or
$\{13k+1,13k+7,13k+11\}$.

Subcase 1.1. $A\cap S_1=\{13k+1,13k+5,13k+11\}$. If $3\nmid
13k+1$, then $13k+1,13k+4,13k+5\in A$ are pairwise coprime. If
$3\nmid 13k+5$, then $13k+1,13k+2,13k+5\in A$ are pairwise
coprime.

Subcase 1.2. $A\cap S_1=\{13k+1,13k+7,13k+11\}$. If $3\nmid
13k+7$, then $13k+7,13k+10,13k+11\in A$ are pairwise coprime. If
$3\nmid 13k+11$, then $13k+7,13k+8,13k+11$ are pairwise coprime.

Case 2. $|A\cap S_1|=4$. It follows that $|A\cap S_2|=5$.

Subcase 2.1. $13k+1\not\in A$. It follows that $A\cap S_1$ has two
pairs of consecutive odd terms. Assume that $a,a+2,b,b+2\in A$,
where $a,b\in S_1$. By $|A\cap S_2|=5$, we have $A\cap
\{a+1,b+1\}\ge 1$. Hence $a,a+1,a+2\in A$ or $b,b+1,b+2\in A$ are
pairwise coprime.

Subcase 2.2. $13k+1\in A$. By the arguments above, we may assume
that $A$ does not contain two pairs of consecutive odd terms and
$$|A\cap \{13k+1,13k+3,13k+5\}|\le 2,~|A\cap \{13k+7,13k+9,13k+11\}|\le
2,$$$$|A\cap \{13k+1,13k+5,13k+9\}|\le 2,~|A\cap
\{13k+3,13k+7,13k+11\}|\le 2.$$ Now we only need to consider the
case $$A\cap S_1=\{13k+1,13k+5,13k+7,13k+11\}.$$ If $13k+6\in A$,
then $13k+5,13k+6,13k+7$ are pairwise coprime. Hence we assume
that $13k+6\not\in A$, and so $(S_2\setminus \{13k+6\})\subseteq
A$. If $3\nmid 13k+5$, then $13k+5,13k+7,13k+8$ are pairwise
coprime. If $3\nmid 13k+7$, then $13k+4,13k+5,13k+7$ are pairwise
coprime.
\end{proof}

\section{Proof of Theorem 2 when $k = 1$}

\begin{prop} Let $A\in A_1$ and $a\in A\setminus F_1$. Then for any positive integer $n$,
$$|A(n)|\le \frac{a-1}{2a}n+\frac{3}{2}.$$\end{prop}

\begin{proof} Clearly, $2\nmid a$. For any nonnegative integer $k$, we have $$|A\cap
\{2ka+2j-1,2ka+2j\}|\le 1,\quad j=1,2,\ldots,a.$$ Furthermore,
since $a$, $2ka+a-2$, $2ka+a-1$ are pairwise coprime, it follows
that $|A\cap \{2ka+a-2,2ka+a-1\}|=0$. Hence, for any nonnegative
integer $k$, we have $|A\cap \{2ka+1,2ka+2,\ldots,2ka+2a\}|\le
a-1$.

Let $n=2aq_1+2q_2+r$, where $0\le q_2\le a-1$ and $0\le r\le 1$.
Then $|A(n)|\le q_1(a-1)+q_2+r$. Now it is enough to prove that
$$q_1(a-1)+q_2+r\le \frac{a-1}{2a}(2aq_1+2q_2+r)+\frac{3}{2}.$$
That is, $\frac{q_2}{a}+\frac{a+1}{2a}r\le \frac{3}{2}$. It is
clear that $$\frac{q_2}{a}+\frac{a+1}{2a}r\le
\frac{q_2}{a}+\frac{a+1}{2a}\le
\frac{a-1}{a}+\frac{a+1}{2a}=\frac{3a-1}{2a}<\frac{3}{2}.$$
Therefore, $|A(n)|\le \frac{a-1}{2a}n+\frac{3}{2}$.

\end{proof}

\section{Proof of Theorem 2 when $k = 2$}

\begin{prop} Let $A\in A_2$ and $a\in A\setminus F_2$. Then for any positive
integer $n$, $$|A(n)|\le \frac{4a-2}{6a}n+\frac{11}{3}.$$\end{prop}

\begin{proof}There exist two integers $n_1,n_2\in \{0,1,\ldots, a-1\}$
such that $a\mid 6n_1-1,~a\mid 6n_2+7$. It is clear that
$n_1\not=n_2$. Otherwise, we have $a\mid 8$, a contradiction. Now
we shall prove that, for $i=1,2$, if $|A\cap
\{6n_i+1,6n_i+2,\ldots,6n_i+6\}|\ge 4$, then $A$ contains $3$
pairwise coprime integers.

Case 1. $a\mid 6n_1-1$. By $|A\cap
\{6n_1+1,6n_1+2,\ldots,6n_1+6\}|\ge 4$, we have $|A\cap
\{6n_1+1,6n_1+2,6n_1+3,6n_1+5\}|\ge 2$. We choose two elements
$c,d\in A\cap \{6n_2+1,6n_2+2,6n_2+3,6n_2+5\}$. Since
$6n_2-1,6n_2+1,6n_2+2,6n_2+3,6n_2+5$ are pairwise coprime, it
follows that $a,c,d\in A$ are pairwise coprime.

Case 2. $a\mid 6n_2+7$. By $|A\cap
\{6n_2+1,6n_2+2,\ldots,6n_2+6\}|\ge 4$, we have $|A\cap
\{6n_2+1,6n_2+3,6n_2+4,6n_2+5\}|\ge 2$. We choose two elements
$e,f\in A\cap \{6n_2+1,6n_2+3,6n_2+4,6n_2+5\}$. Since
$6n_2+1,6n_2+3,6n_2+4,6n_2+5,6n_2+7$ are pairwise coprime, it
follows that $a,e,f\in A$ are pairwise coprime.

Hence, by $A\in A_2$, for $i=1,2$, we have $|A\cap
\{6n_i+1,6n_i+2,\ldots,6n_i+6\}|\le 3$. Let $n=6aq_1+6q_2+r$,
where $0\le q_2<a$ and $0\le r\le 5$. Then ~$|A(n)|\le
(4a-2)q_1+4q_2+r$. Now we need to prove
$$(4a-2)q_1+4q_2+r\le \frac{4a-2}{6a}(6aq_1+6q_2+r)+\frac{11}{3}.$$
That is, $$\frac{(2a+2)r+12q_2}{6a}\le \frac{11}{3}.$$ Clearly we
have
$$\frac{(2a+2)r+12q_2}{6a}\le  \frac{5(2a+2)+12(a-1)}{6a}\le
\frac{22a}{6a}=\frac{11}{3}.$$\end{proof}

\section{Proof of Theorem 2 when $k = 3$}

\begin{prop} Let $A\in A_3$ and $a\in A\setminus F_3$. Then for any positive integer $n$, $$|A(n)|\le \frac{22a-4}{30a}n+\frac{176}{15}.$$\end{prop}

\begin{proof} Since $a\not\in F_3$, it follows that $(a,30)=1$.

We first consider the case $a=7$. Let $n=7s+r,~0\le r\le 6$. Then,
for any nonnegative integer $k$, by Lemma \ref{lem7}, we have
$|A\cap\{7k+1,7k+2,\ldots,7k+7\}|\le 5$. Hence $$|A(n)|\le
5s+r=\frac{22\cdot 7-4}{30\cdot 7}\cdot
7s+r=\frac{22a-4}{30a}(7s+r)+\frac{8a+4}{30a}r\le
\frac{22a-4}{30a}n+\frac{176}{15}.$$

Next we consider the case $a=13$. Let $n=13s+r,~0\le r\le 12$. For
any positive integer $k$, by Lemma \ref{lem13}, we have
$|A\cap\{13k+1,13k+2,\ldots,13k+13\}|\le 9$. It also follows from
Lemma \ref{lem13} that $|A\cap \{1,2,\ldots,13\}|\le 10$. Hence
$|A(n)|\le 9s+r\le \frac{22\times 13-4}{30\times 13}\cdot
(13s)+r=\frac{22a-4}{30a}(13s+r)+\frac{8a+4}{30a}r\le
\frac{22a-4}{30a}n+\frac{176}{15}.$

Now we may assume that $a\not=7,13$. Then there exist four
distinct integers $n_1,n_2,n_3,n_4\in \{0,1,\ldots, a-1\}$ such
that $a\mid 30n_1-1,~a\mid 30n_2+31,~a\mid 30n_3-11,~a\mid
30n_4+41$. Now we shall prove that, for $i=1,2,3,4$, if $|A\cap
\{30n_i+1,30n_i+2,\ldots,30n_i+30\}|=22$, then $A$ contains $4$
pairwise coprime integers. By Lemma \ref{lem55}, it suffices to
prove the case that $$|(\mathbb{Z}^{+}\setminus A)\cap (30k+
\{1,3,4,5,7,8,11,13\})|=4,$$
$$|(\mathbb{Z}^{+}\setminus A)\cap
(30k+\{17,19,22,23,25,26,27,29\})|=4.$$ In this case, we have
$$(30k+\{2,6,9,10,12,14,15\})\subseteq A,$$
$$(30k+\{16,18,20,21,24,28,30\})\subseteq A.$$

Case 1. $a\mid 30n_1-1$.

Write $M:=30n_1$. If $A\cap \{M+4,M+8\}=\emptyset$, then $|A\cap
(M+\{1,3,5,7,11,13\})|= 4$, and so $A$ contains $4$ pairwise
coprime integers.

If $M+4\in A$ and $M+8\not\in A$, then $M+11\in A$ and $7\mid
M+4$. Otherwise, $M+4$ and three integers in $A\cap (M+
\{1,3,5,7,11,13\})$ are pairwise coprime. By $7\mid M+4$, it
follows that $7\nmid M-1$. Hence $a$ and three integers in $A\cap
(M+\{1,3,5,7,11,13\})$ are pairwise coprime.

If $M+8\in A$ and $M+4\not\in A$, then $M+1\in A$ and $7\mid M+8$.
Otherwise, $M+8$ and three integers in $A\cap (M+
\{1,3,5,7,11,13\})$ are pairwise coprime. Now we have $7\nmid
M-1$. Hence $a$ and three integers in $A\cap
(M+\{1,3,5,7,11,13\})$ are pairwise coprime.

Now we suppose that $M+4\in A$ and $M+8\in A$. If $M+13\not\in A$,
then two integers in $A\cap (M+\{1,3,5,7,11\})$, $a$ and one
integer in $\{M+4,M+8\}$ which is not a multiple of $7$ are
pairwise coprime. Hence $M+13\in A$.

If there exists an integer $b\in \{M+1,M+5,M+7\}$ and $b\in A$,
then $b,M+4,M+9,M+13$ are pairwise coprime.

If $M+11\in A$, then $M+8,M+9,M+11,M+13$ are pairwise coprime.

Therefore $$A\cap (M+\{1,3,4,5,7,8,11,13\})=(M+\{3,4,8,13\}).$$

If $7\nmid a$, then $a,M+3,M+13$ and the integer in $\{M+4,M+8\}$
which is not a multiple of $7$ are pairwise coprime. Hence now we
assume that $7\mid a$.

Since $|(\mathbb{Z}^{+}\setminus A)\cap (30k+
\{17,19,22,23,25,26,27,29\})|=4$, it follows that
$|(M+\{17,19,23,25,29\})\cap A|\ge 1.$

If $M+17\in A$, then $a,M+3,M+8,M+17\in A$ are pairwise coprime.

If $M+19\in A$, then $a,M+3,M+4,M+19\in A$ are pairwise coprime.

If $M+23\in A$, then $a,M+3,M+8,M+23\in A$ are pairwise coprime.

If $M+25\in A$, then $M+4,M+9,M+13,M+25\in A$ are pairwise
coprime.

If $M+29\in A$, then $a,M+4,M+9,M+29\in A$ are pairwise coprime.

Hence, $|A\cap \{M+1,M+2,\ldots,M+30\}|\le 21$.

Case 2. $a\mid 30n_2+31$.

This case is nearly the same as Case 1, we only need to replace
$30n_1+i$ by $30n_2+30-i$ for $i=1,2,\ldots,29$ and replace
$30n_1+30$ by $30n_2+30$. The proofs are the same.

Case 3. $a\mid 30n_3-11$.

Write $M:=30n_3$. If $A\cap \{M+4,M+8\}=\emptyset$, then $|A\cap
(M+\{1,3,5,7,11,13\})|= 4$, and so $A$ contains $4$ pairwise
coprime integers.

If $M+4\in A$ and $M+8\not\in A$, then $M+11\in A$ and $7\mid
M+4$. Otherwise, $M+4$ and three integers in $A\cap (M+
\{1,3,5,7,11,13\})$ are pairwise coprime. By $7\mid M+4$, we have
$7\nmid M-11$. Hence $a$ and three integers in $A\cap
(M+\{1,3,5,7,11,13\})$ are pairwise coprime.

If $M+8\in A$ and $M+4\not\in A$, then $M+1\in A$ and $7\mid M+8$.
Otherwise, $M+8$ and three integers in $A\cap (M+
\{3,5,1,7,11,13\})$ are pairwise coprime. By $7\mid M+8$, we have
$7\nmid M-11$. If $M+11\not\in A$, then $a$ and three integers in
$A\cap (M+\{1,3,5,7,13\})$ are pairwise coprime. Now we assume
that $M+11\in A$. If $M+3\not\in A$, then $M+9$ and $A\cap
(M+\{1,5,7,11,13\})$ are pairwise coprime. If $M+3\in A$, then
$M+1,M+2,M+3,M+11\in A$ are pairwise coprime.

Now we suppose that $M+4\in A$ and $M+8\in A$. If $M+11\not\in A$,
choose an integer $b\in A\cap (M+\{1,5,7,13\})$, then
$a,b,M+4,M+9\in A$ are pairwise coprime. Next we assume that
$M+11\in A$.

If there exists an integer $b\in A$ and $b\in \{M+5,M+7,M+13\}$,
then $b,M+8,M+9,M+11$ are pairwise coprime.

If $M+1\in A$, then $M+1,M+9,M+10,M+11\in A$ are pairwise coprime.

If $M+3\in A$ and $11\nmid a$, then $a,M+9,M+11,M+14$ are pairwise
coprime.

Finally we assume that $M+3\in A$ and $11\mid a$.

Subcase 1. $M+22\in A$ and $M+26\not\in A$.

If $M+29\not\in A$, then $M+22$ and three integers in $A\cap
(M+\{17,19,23,25,27\})$ are pairwise coprime.

If $M+29\in A$, then $M+11,M+9,M+14,M+29\in A$ are pairwise
coprime.

Subcase 2. $M+22\not\in A$ and $M+26\in A$.

If $M+19\not\in A$, then $M+26$ and three integers in $A\cap
(M+\{17,23,25,27,29\})$ are pairwise coprime.

If $M+19\in A$, then $a,M+19,M+9,M+4\in A$ are pairwise coprime.

Subcase 3. $M+22\in A$ and $M+26\in A$.

Since $|(\mathbb{Z}^{+}\setminus A)\cap (M+
\{17,19,22,23,25,26,27,29\})|=4$, we have $$|A\cap
(M+\{17,19,23,25,29\})\ge 1.$$

If $M+17\in A$, then $M+8,M+9,M+11,M+17\in A$ are pairwise
coprime.

If $M+19\in A$, then $a,M+4,M+9,M+19\in A$ are pairwise coprime.

If $M+23\in A$, then $M+3,M+8,M+11,M+23\in A$ are pairwise
coprime.

If $M+25\in A$,  by $11\mid M-11$, then $11\nmid M+14$ and
$a,M+9,M+14,M+25\in A$ are pairwise coprime.

If $M+29\in A$, then $M+9,M+11,M+14,M+29\in A$ are pairwise
coprime.

Hence, $|A\cap \{M+1,M+2,\ldots,M+30\}|\le 21$.

Subcase 4. $M+22\notin A$ and $M+26\notin A$. It follows that $|A\cap
(M+\{17,19,23,25,27,29\})|= 4$, and so $A$ contains $4$ pairwise
coprime integers.

Case 4. $a\mid 30n_4+41$.

This case is nearly the same as Subcase 2, we only need to replace
$30n_3+i$ by $30n_4+30-i$ for $i=1,2,\ldots,29$ and replace
$30n_3+30$ by $30n_4+30$. The proofs are the same.

Therefore, for any positive integer $q$, we have $|A(n)\cap
\{1,2,\ldots,30aq\}|\le (22a-4)q$ and $|A(n)\cap
\{30q+1,30q+2,\ldots,30q+30\}|\le 22$.

Now let $n=30aq_1+30q_2+r,~0\le r<30,~0\le q_2\le a-1$. Then
$|A(n)|\le (22a-4)q_1+22q_2+r$, where $0\le r<30$ and $0\le q_2\le
a-1$. Hence it suffices to prove that
\begin{eqnarray*}(22a-4)q_1+22q_2+r\le
\frac{22a-4}{30a}(30aq_1+30q_2+r)+\frac{176}{15}.
\end{eqnarray*}
That is, $\frac{4q_2}{a}+r-\frac{22a-4}{30a}r\le \frac{176}{15}$.
It is clear that $$\frac{4q_2}{a}+r-\frac{22a-4}{30a}r\le
\frac{4(a-1)}{a}+29-\frac{22a-4}{30a}\times 29<4+29-\frac{22\times
29}{30}=\frac{176}{15}.$$
\end{proof}

\section{Proof of Theorem 3}

In the first step we prove the first part of Theorem 3. Let $A\in
C_1(n)$, $A\not\subseteq E_1(n)$ and let $a$ be the smallest odd
integer of $A$. We may assume that $a$ is squarefree, i.e., $a =
q_{1} q_2\cdots q_{t}$, where $3 \le q_{1} < \dots{} < q_{t}$ are
primes. It is clear that two consecutive integers are coprime, and
 the conditions $(m,a) = (m+1,a) = 1$ means that $m$ satisfies the following linear congruences
$m \equiv a_{i} \bmod{q_{i}}$ and $m + 1 \equiv a_{i}+1
\bmod{q_{i}}$, where $1 \le a_{i} \le q_{i} - 2$. The number of
such congruences is $ \prod_{i=1}^{t}(q_{i} - 2).$

 Thus we have
\begin{eqnarray*}
&&|\{m: 1 \le m \le a, (m,a) = (m+1,a)= 1\}| = \prod_{i=1}^{t}(q_{i} - 2)\\
&= &a \cdot \prod_{i=1}^{t}\Big(1 -
\frac{1}{q_{i}}\Big)^{2}\cdot\prod_{i=1}^{t}\frac{1-2/q_{i}}{(1-1/q_{i})^{2}}
= a\cdot \prod_{i=1}^{t}\Big(1-\frac{1}{q_{i}}\Big)^{2} \cdot \prod_{i=1}^{t}\frac{q^{2}_{i}-2q_{i}}{q_{i}^2-2q_{i}+1}\\
&\ge& a\cdot \prod_{i=1}^{t} \Big(1-\frac{1}{q_{i}}\Big)^{2} \cdot
\prod_{p \ge 3}\Big(1 - \frac{1}{(p-1)^{2}}\Big) \gg  a \cdot
\prod_{i=1}^{t}\Big(1-\frac{1}{q_{i}}\Big)^{2}\\
&=& a \cdot \frac{\varphi(a)^{2}}{a^2} \gg \frac{1}{a} \cdot
\Big(\frac{a}{\log \log a}\Big)^{2} = \frac{a}{(\log\log a)^{2}}.
\end{eqnarray*}
It follows that
\begin{eqnarray*}
&&|\{m: 1 \le m \le n, (m,a) = (m+1,a) = 1\}| \\
&\ge& |\{m: 1 \le m \le \Big[\frac{n}{a}\Big]a, (m,a) = (m+1,a) = 1\}| \\
&\gg& \frac{a}{(\log\log a)^{2}} \cdot \Big[\frac{n}{a}\Big] \gg \frac{n}{(\log\log n)^{2}}.
\end{eqnarray*}

We assume that there are $v$ even integers and $w$ odd integers in
the set $$\{m: 1 \le m \le n, (m,a) = (m+1,a) = 1\}.$$ Obviously,
 $v \gg \frac{n}{(\log\log n)^{2}}$
or $w \gg \frac{n}{(\log\log n)^{2}}$.

If
 $v \gg \frac{n}{(\log\log n)^{2}}$, then we may choose in $A$ at most one integer from each pair of $(1,2), \dots{}, (2[n/2]+1, 2[n/2] + 2)$.
Thus we have
\begin{eqnarray*}
|A|= \sum_{k=1}^{\lfloor \frac{n}{2}\rfloor+1}|A\cap
\{2k-1,2k\}|\le \left\lfloor \frac{n}{2}\right\rfloor+1-v\le
\frac{n}{2} - c_{1}\frac{n}{(\log\log n)^{2}},
\end{eqnarray*}
where $c_1$ is an absolute positive constant.

If $w \gg \frac{n}{(\log\log n)^{2}}$, then we may chose in $A$ at
most one integers from each pair of $(2,3), \dots{}, (2[n/2],
2[n/2]+1)$. Thus we have
\begin{eqnarray*}
|A|\le 1+\sum_{k=1}^{\lfloor \frac{n}{2}\rfloor}|A\cap
\{2k,2k+1\}|\le 1+\left\lfloor \frac{n}{2}\right\rfloor-w\le
\frac{n}{2} - c_{2}\frac{n}{(\log\log n)^{2}},
\end{eqnarray*}
where $c_2$ is also an absolute positive constant.

This completes the proof of the first part of Theorem 3.

In the next step we prove the second part of Theorem 3. First we
need three lemmas in the following.

\begin{lem} If $A \in C_{2}(n)$, then for every positive integer $k$ we have
\[
|\{6k, 6k + 1, 6k + 2, 6k + 3, 6k + 4, 6k + 5\} \cap A| \le 4.
\]
\end{lem}

\begin{proof} Assume contrary that $|\{6k, 6k + 1, 6k + 2, 6k + 3, 6k + 4, 6k + 5\} \cap A| \ge 5$. If $6k + 1, 6k + 2, 6k + 3 \in A$, then
$6k + 1, 6k + 2, 6k + 3$ are pairwise coprime, a contradiction. If
$6k + 3, 6k + 4, 6k + 5 \in A$, then $6k + 1, 6k + 2, 6k + 3$ are
also pairwise coprime, a contradiction again. Now we may assume
$6k + 3 \notin A$. Then we have $6k + 1, 6k + 2, 6k + 5 \in A$,
but $6k + 1, 6k + 2, 6k + 5$ are obviously pairwise coprime, which
is absurd.

The proof of the Lemma is completed.
\end{proof}

\begin{lem}\label{lem66} Let $a$ be a positive integer with $(a,6)=1$ and
$a\ge \prod_{i=3}^{12}p_i$. Then we have
\begin{eqnarray*}
|\{m: 1 \le m \le a, 6\mid m~\text{and}~ (m+i,a) = 1 ~\text{for}~i=0,1,2,3,5\}| \gg \frac{a}{(\log\log a)^{5}}.
\end{eqnarray*}
\end{lem}

\begin{proof}
Let $a = q_{1}q_2 \cdots{} q_{t}$, where $5 \le q_{1} < q_{2} <
\cdots{} < q_{t}$ are primes. Suppose that $t = 1$. Take $m = 6,
12, \dots{}, 6\Big(\Big[\frac{a}{6}\Big] - 1\Big)$. Then we have
$(m+i,a) = 1 ~\text{for}~i=0,1,2,3,5$, and so
\[
|\{m: 1 \le m \le a, 6\mid m ~\text{and}~ (m+i,a) = 1 ~\text{for}~i=0,1,2,3,5\}| \ge \Big[\frac{a}{6}\Big] - 1 \gg \frac{a}{(\log\log a)^{5}}.
\]

Now we assume that $t \ge 2$. Define the set $S$ by
\[
S = \left\{s: 1 \le s \le \frac{a}{q_{t}}, \left(s+i, \frac{a}{q_{t}}\right) = 1 ~\text{for}~i=0,1,2,3,5\right\}
\]
Then we have
\begin{eqnarray*}
&&\{m: 1 \le m \le a, (m+i,a) = 1 ~\text{for}~i=0,1,2,3,5\}\\
&=& \bigcup_{s \in S}\left\{s + l\cdot \frac{a}{q_{t}}: 0 \le l
\le q_{t} - 1, (s + l\cdot \frac{a}{q_{t}}+i,q_t) = 1
~\text{for}~i=0,1,2,3,5 \right\}: = \bigcup_{s \in S}S_{s}.
\end{eqnarray*}
It is easy to see that, for a fixed $s \in S$ and $i\in
\{0,1,2,3,5\}$, the integers $s + l\cdot \frac{a}{q_{t}}+i~(0 \le
l \le q_{t} - 1)$ form a complete residue system modulo $q_{t}$.
Since $q_{t} \ge 7$, we have $|S_{s}| = q_{t} - 5$ for every $s
\in S$. It follows from $(\frac{a}{q_{t}}, 6) = 1$ that
\[
\left|\left\{l: 0 \le l \le q_{t} - 1, 6\mid s + l\cdot \frac{a}{q_{t}}\right\}\right| \ge \Big[\frac{q_{t}}{6}\Big].
\]
Then we have
\[
\left|\left\{l: 0 \le l \le q_{t} - 1, 6\mid s + l\cdot \frac{a}{q_{t}}~\text{and}~  \left(s + l\cdot \frac{a}{q_{t}}+i,a\right) = 1~\text{for}~i=0,1,2,3,5
\right\}\right| \ge \Big[\frac{q_{t}}{6}\Big] - 5.
\]
It follows that
\begin{eqnarray*}
&&\left|\left\{l: 0 \le l \le q_{t} - 1, 6\mid s + l\cdot \frac{a}{q_{t}}~\text{and}~  \left(s + l\cdot \frac{a}{q_{t}}+i,a\right) = 1~\text{for}~i=0,1,2,3,5
\right\}\right|\\
&\ge& \frac{\Big[\frac{q_{t}}{6}\Big] - 5}{q_{t} - 5} \cdot
\left|\left\{l: 0 \le l \le q_{t} - 1~\text{and}~ \left(s + l\cdot
\frac{a}{q_{t}}+i,a\right) = 1~\text{for}~i=0,1,2,3,5
\right\}\right|.
\end{eqnarray*}
Then we obtain
\begin{eqnarray*}
&&|\{m: 1 \le m \le a, 6\mid m ~\text{and}~ (m+i,a) = 1 ~\text{for}~i=0,1,2,3,5\}|\\
&\ge& \frac{\Big[\frac{q_{t}}{6}\Big] - 5}{q_{t} - 5} \cdot |\{m: 1 \le m \le a ~\text{and}~ (m+i,a) = 1 ~\text{for}~i=0,1,2,3,5\}|.
\end{eqnarray*}
If $q_{1} = 5$, then $(m,a) = (m+1,a) = (m+2,a) = (m+3,a) =
(m+5,a) = 1$ means that $m$ satisfies the following linear
congruences $m \equiv 1 \bmod{5}$ and $m \equiv a_{i}
\bmod{q_{i}}$, where $a_{i}\in \{1, 2, \dots{}, q_{t} - 6,q_{t} -
4\}$ for $i=2,3,\ldots,t$. The number of such congruences is
$\prod_{i=2}^{t}(q_{i} - 5)$. Hence we have
\begin{eqnarray*}
&&|\{m: 1 \le m \le a ~\text{and}~ (m+i,a) = 1 ~\text{for}~i=0,1,2,3,5\}| = \prod_{i=2}^{t}(q_{i} - 5)\\
&=& \frac{1}{5}\left(\prod_{i=1}^{t}q_{i}\right)\left(\prod_{i=2}^{t}\frac{1-\frac{5}{q_{i}}}{(1-\frac{1}{q_{i}})^{5}}\right)\left(\prod_{i=1}^{t}\left(1 - \frac{1}{q_{i}}\right)^{5}\right)\frac{1}{(1-\frac{1}{5})^{5}}\\
&\ge& \frac{\prod_{p \ge 7}\frac{(1-\frac{5}{p})}{(1-\frac{1}{p})^{5}}}{5(1-\frac{1}{5})^{5}}\cdot a \cdot
\frac{\varphi(a)^{5}}{a^{5}} \gg \frac{a}{(\log\log a)^{5}}.
\end{eqnarray*}

If $q_{1} \ge 7$, the proof is the similar. For the sake of
completeness, we present it here.

If $q_{1} \ge 7$, then $(m,a) = (m+1,a) = (m+2,a) = (m+3,a) =
(m+5,a) = 1$ means that $m$ satisfies the following linear
congruences $m \equiv a_{i} \bmod{q_{i}}$, where $a_{i}\in \{ 1,
2, \dots{}, q_{t} - 6,q_{t} - 4\}$ for $i=1,2\ldots,t$. The number
of such congruences is $\prod_{i=1}^{t}(q_{i} - 5)$. Hence we
obtain
\begin{eqnarray*}
&&\prod_{i=1}^{t}(q_{i} - 5) = \left(\prod_{i=1}^{t}q_{i}\right)\left(\prod_{i=1}^{t}\frac{1-5/q_{i}}{(1-1/q_{i})^{5}}\right)\left(\prod_{i=1}^{t}(1-\frac{1}{q_{i}})^{5}\right)\\
&\ge& \prod_{p \ge 7}\frac{(1-5/p)}{(1-1/p)^{5}}\cdot a \cdot
\frac{\varphi(a)^{5}}{a^{5}}
\gg \frac{a}{(\log\log a)^{5}}.
\end{eqnarray*}

Since $a\ge \prod_{i=3}^{12}p_i$, it follows that $q_{t} \ge 37$,
and so
\[
\frac{\Big[\frac{q_{t}}{6}\Big] - 5}{q_{t} - 5} \ge
\frac{\Big[\frac{41}{6}\Big] - 5}{41 - 5}=\frac{1}{36},
\]
which completes the proof.
\end{proof}

\begin{lem} Let $A \in C_{2}(n)$, $a \in A$ and $k$ be a positive integer with
$(6k,a) = (6k + 1,a) = (6k + 2,a) = (6k + 3,a) = (6k + 5,a) = 1$, then we have
\[
|\{6k, 6k + 1, 6k + 2, 6k + 3, 6k + 4, 6k + 5\} \cap A| \le 3.
\]
\end{lem}
\begin{proof}
Assume contrary that
\[
|\{6k, 6k + 1, 6k + 2, 6k + 3, 6k + 4, 6k + 5\} \cap A| \ge 4.
\]
Then it is clear that
\[
|\{6k + 1, 6k + 2, 6k + 3, 6k + 5\} \cap A| \ge 2.
\]
Obviously, if we choose two elements from the set $\{6k + 1, 6k +
2, 6k + 3, 6k + 5\} \cap A$, then these elements and $a$ are
pairwise coprime, which is absurd.
\end{proof}

Now we are ready to prove the second case of Theorem 3. Let $a$ be the smallest element of $A$ with $(a, 6) = 1$.
 According to Lemma 5 and Lemma 7 we have
\begin{eqnarray*}
|A|& =& |\{1,2,3,4,5\} \cap A| + \sum_{k = 1}^{[n/6]}|\{6k, 6k + 1, 6k + 2, 6k + 3, 6k + 4, 6k + 5\} \cap A|\\
&\le& 5 + 4\cdot\left[\frac{n}{6}\right] - |\{m: 1 \le m \le n,
6\mid m ~\text{and}~ (m+i,a) = 1 ~\text{for}~i=0,1,2,3,5\}|.
\end{eqnarray*}
Thus it is enough to prove that
\[
|\{m: 1 \le m \le a, 6\mid m ~\text{and}~ (m+i,a) = 1 ~\text{for}~i=0,1,2,3,5\}| \gg \frac{n}{(\log\log n)^{5}}
\]
We distinguish two cases.

Case 1. $a < \prod_{i=3}^{12}p_i$. There exists a $0 \le l \le 5$
such that $6 \mid l\cdot a + 1$. We choose $m = l\cdot a + 1$.
Then we have $(m,a) = (l\cdot a+1,a) = 1$, $(m+1,a) = (l\cdot
a+2,a) = 1$, $(m+2,a) = (l\cdot a+3,a) = 1$, $(m+3,a) = (l\cdot a
+ 4,a)= 1$, $(m+5,a) = (l\cdot a + 6,a)= 1$. It follows that
\begin{eqnarray*}
&&|\{m: 1 \le m \le n, 6\mid m ~\text{and}~ (m+i,a) = 1 ~\text{for}~i=0,1,2,3,5\}|\\
&\ge& \left|\left\{6aj + la + 1: 0 \le j \le
\left[\frac{n}{6a_{0}}\right] - 1\right\}\right| =
\left[\frac{n}{6a_{0}}\right] \gg  \frac{n}{(\log\log n)^{5}}.
\end{eqnarray*}

Case 2. $a \ge  \prod_{i=3}^{12}p_i$. By Lemma \ref{lem66}, we
have
\begin{eqnarray*}
&&|\{m: 1 \le m \le a, 6\mid m ~\text{and}~ (m+i,a) = 1 ~\text{for}~i=0,1,2,3,5\}|\\
&\ge& \left|\bigcup_{0 \le j \le \left[\frac{n-a}{6a}\right]}(6ja + \{m: 1 \le m \le a, 6\mid m ~\text{and}~ (m+i,a) = 1 ~\text{for}~i=0,1,2,3,5\})\right|\\
&\gg& \frac{n}{6a} \cdot \frac{a}{(\log\log a)^{5}} \gg \frac{n}{(\log\log n)^{5}}.
\end{eqnarray*}
The proof of Theorem 3 is completed.

\section{Proof of Theorem 4}

For any integer $n\ge \prod_{i=1}^{k+2}p_i$, there exists an
integer $l\ge k+1$ such that $$\prod_{i=1}^{l+1}p_{i} \le n <
\prod_{i=1}^{l+2}p_{i}.$$ We define the set \[ A = \{m: m \le n,
\text{there exists}~ p_{i},p_j~\text{with}~i \le k< j \le
l~\text{such that}~p_ip_{j} \mid m \} \cup \{p_{k+1}p_{k+2}\cdots
p_{l}\}.
\]
Obviously, $A \in C_{k}(n)$ and $A \not\subseteq E_k(n)$. It is
easy to see that
\[
|A \cap [1, p_{1}p_2\cdots p_{l}]| = \Big(p_{1}p_2\cdots p_{k} -
\varphi(p_{1}p_2\cdots  p_{k})\Big) \cdot
\Big(p_{k+1}p_{k+2}\cdots  p_{l} - \varphi(p_{k+1}p_{k+2}\cdots
p_{l})\Big) + 1.
\]
It follows that
\begin{eqnarray*}
|A| &\ge& \left|A \cap \left[1, \left[\frac{n}{p_{1}p_2\cdots  p_{l}}\right]p_{1}p_2\cdots p_{l}\right]\right|\\
&\ge& \left[\frac{n}{p_{1}p_2\cdots  p_{l}}\right]
\left(p_{1}p_2\cdots p_{k} - \varphi(p_{1}p_2\cdots  p_{k})\right)
\cdot \left(p_{k+1}p_{k+2}\cdots  p_{l} -
\varphi(p_{k+1}p_{k+2}\cdots p_{l})\right).
\end{eqnarray*}
We give an upper estimation to the fractional part of
 $\frac{n}{p_{1}p_2\cdots p_{l}}$. It is easy to see from the definition of $l$ that $p_{l+1} \sim \log n$ and
 therefore
\[
\Big\{\frac{n}{p_{1}p_2\cdots  p_{l}}\Big\} \le 1 \le c \cdot
\frac{p_{l+1}}{\log n}  \le c \cdot \frac{n}{p_{1}p_2\cdots p_{l}}
\cdot \frac{1}{\log n},
\]
where $c$ is an absolute positive constant. Thus we have
\[
\Big[\frac{n}{p_{1}p_2\cdots  p_{l}}\Big] =
\frac{n}{p_{1}p_2\cdots  p_{l}} - \Big\{\frac{n}{p_{1}p_2\cdots
 p_{l}}\Big\} \ge \frac{n}{p_{1}p_2\cdots
p_{l}}\Big(1 - \frac{c}{\log n}\Big).
\]
Clearly,
\[
\varphi(p_{k+1}p_{k+2}\cdots p_{l}) = \frac{p_{1}p_2\cdots
p_{k}}{\varphi(p_{1}p_2\cdots  p_{k})}\cdot
\frac{\varphi(p_{1}p_2\cdots  p_{l})}{p_{1}p_2\cdots p_{l}}\cdot
p_{k+1}p_{k+2}\cdots  p_{l} < c_{k}\frac{p_{k+1}p_{k+2}\cdots
 p_{l}}{\log\log n}.
\]

It follows that
\begin{eqnarray*}
|A| &\ge& \frac{n}{p_{1}p_2\cdots p_{l}}\Big(1 - \frac{c}{\log
n}\Big)\cdot \Big(p_{1}p_2\cdots p_{k} - \varphi(p_{1}p_2\cdots
p_{k})\Big)
\cdot p_{k+1}p_{k+2}\cdots p_{l}\Big(1 - \frac{c_{k}}{\log\log n}\Big)\\
&\ge& \frac{n}{p_{1}p_2\cdots p_{k}}\Big(p_{1}p_2\cdots  p_{k} -
\varphi(p_{1}p_2\cdots p_{k})\Big) - \frac{2c_{k}n}{\log\log n},
\end{eqnarray*}
where $c_{k}$ is an absolute positive constant depending only on
$k$. Obviously,
\[
|E_k(n)| = \frac{n}{p_{1}p_2\cdots p_{k}}\Big(p_{1}p_2\cdots p_{k}
- \varphi(p_{1}p_2\cdots  p_{k})\Big) + O_{k}(1).
\]
Hence
\[
|E_k(n)|-|A|\ll_k \frac{n}{\log\log n}.
\]

This completes the proof of Theorem 4.

\bigskip

\section{Acknowledgement} The authors would like to thank J\'anos
Pintz for the valuable discussions about Theorem 1. This work was done during
the third author visiting to Budapest University of Technology and Economics.
He would like to thank Dr. S\'andor Kiss and Dr. Csaba S\'{a}ndor for their
warm hospitality.

\end{document}